\documentclass[a4paper,12pt]{article}
\usepackage[english]{babel}
\usepackage{hyperref}
\usepackage{indentfirst}
\usepackage{graphicx}
\usepackage{amsfonts,amssymb,amsthm,amsmath}
\newcounter{tmp}
\newcommand{\ir}{[0,T]\times \mathbb{R}^n}

\newcommand{\gr}[1]{{\rm Gr}_{#1}}
\newcommand{\grp}[1]{{\rm Gr}_{{#1}+1}}
\newcommand{\dnr}{\Delta^N\times \mathbb{R}^n}
\newtheorem{theorem}{Theorem}[section]
\newtheorem{lemma}{Lemma}[section]
\newtheorem{definition}{Definition}[section]
\newtheorem{remark}{Remark}[section]
\newtheorem{proposition}{Proposition}[section]
\newtheorem{example}{Example}[section]

\begin{document}
\title{Universal Nash Equilibrium Strategies for 
\\ Differential Games}

\author{Yurii Averboukh\footnote{Institute of Mathematics and Mechanics UrB RAS \& Ural Federal University, 16, S.~Kovalevskaya str., Ekaterinburg, 620990, Russia, \texttt{ayv@imm.uran.ru, averboukh@gmail.com}}}
\maketitle 
\begin{abstract}
\noindent The paper is concerned with a two-player nonzero-sum differential game in the case when players are informed about the current position. We consider the game in control with guide strategies first proposed by Krasovskii and Subbotin. The construction of universal strategies is given both for the case of continuous and discontinuous value functions. The existence of a discontinuous value function is established. The continuous value function does not exist in the general case. In addition, we show the example of smooth value function not being a solution of the system of Hamilton--Jacobi equation.\end{abstract}

{\small \textbf{Keywords.} {Nash equilibrium, nonzero-sum differential game, control with guide strategies.}

}

\section{Introduction}
The purpose of this paper is to study Nash equilibria for a two-player deterministic differential game in the case when the players are informed about the present position. We look for the universal equilibrium solution. The term `universal Nash equilibrium strategies' means that the strategies provide the Nash equilibrium at any initial position. The notion of universality generalizes the notion of time consistency, and it is appropriative for the case when the players form their controls stepwise. Generally speaking, in this case the notion of time consistence isn't well-defined.

There are two approaches in the literature dealing with this problem (see \cite{cardal_survey}, and the references therein). The first approach is close to the so-called Folk Theorem for repeated games, and is based on the punishment strategy technique. This technique makes it possible to establish the   existence of Nash equilibrium at the given initial position in the framework of feedback strategies \cite{Cleimenov}, \cite{Kononenko} and in the framework of Friedman strategies  \cite{Tolwinski}. The set of Nash equilibria at the given initial position is characterized in \cite{Chistyakov_Nash}, \cite{Cleimenov}. The infinitesimal version of this characterization is derived in \cite{AverboukhIMM}, \cite{AverboukhAnnals}. In addition, each Nash equilibrium payoff at the given position corresponds to the pair of continuous functions; these functions are stable with respect to auxiliary zero-sum differential games, and their values at the initial position are kept along some trajectory \cite{AverboukhSb}. Note, that in this case the Nash equilibrium strategies are not universal and strongly time consistent.

The key idea of the second approach  is to find a Nash equilibrium payoff as a solution of the system of Hamilton--Jacobi equations  \cite{Basar}, \cite{Case}, \cite{Friedman}. In this case the universal Nash equilibrium can be constructed. In particular, they are strongly time consistent. However, the existence theorem for the system of Hamilton--Jacobi equations is established only for some  cases  of the games in one dimension \cite{Bres2}, \cite{Bres3}, \cite{Cardal_Plaskatz}.


In this paper we consider the Nash equilibrium for deterministic differential games in  control with guide strategies. These strategies was first proposed by Krasovskii and Subbotin for zero-sum differential games \cite{NN_PDG_en}. In the framework of this formalization the player forms his control stepwise. It is assumed that the player measures the state of the system only in the times of control correction. At the time  of control correction  the player estimates the state of the system using on the information about the state of the system at the previous time instants of control correction. Having this estimate and the information about the real state of the system he assigns the control which is used up to the next control correction.

The choice of control with guide strategies is motivated by the following arguments. The universal optimal feedback doesn't exist even for the case of zero-sum differential game \cite{Subbotina}. The universal  solution of zero-sum differential games can be find in the class of feedback strategies depending on the precision parameter \cite{krasovskii_under_lack}, or in the class of control with guide strategies \cite{NN_PDG_en}. However, for the case of  nonzero-sum differential games existing design of Nash equilibria   in the class  feedback strategies depending on the precision parameter doesn't provide the universality.

The paper is organized as follows. In Section~2 we set up the problem, and introduce the control with guide strategies. In Section~3 we construct the Nash equilibrium in the control with guide strategies for the  case of a continuous value function. This function is to satisfy some viability conditions. 
Further in Section 3 the properties of a continuous value function are considered. We give the infinitesimal form of viability  conditions. After we compare the value functions satisfying viability conditions and the solutions of the system of Hamilton--Jacobi equations. The example showing that the continuous value function does not exist in the general case completes Section~3. In Section~4 we generalize the construction of Section~3 for the case of a upper semicontinuous value multifunction. In Section~5 we prove the existence  of a  value multifunction.

\section{Problem Statement}
Let us consider a two-player differential game with the dynamics
\begin{equation}\label{system}
    \dot{x}=f(t,x,u)+g(t,x,v), \ \ t\in [0,T], \ \ x\in\mathbb{R}^n, \ \ u\in P, \ \ v\in Q.
\end{equation} Here $u$ and $v$ are controls of  player I and  player II respectively. Payoffs are terminal. Player I wants to maximize $\sigma_1(x(T))$, whereas player II wants to maximize $\sigma_2(x(T))$. We assume that sets $P$ and $Q$ are compacts,  functions $f$, $g$, $\sigma_1$ and $\sigma_2$ are continuous. In addition, suppose that  functions $f$ and $g$ are Lipschitz continuous with respect to the phase variable and satisfy the sublinear growth condition with respect to $x$.

Denote $$\mathcal{U}:=\{u:[0,T]\rightarrow P\ \ measurable\}, $$
$$\mathcal{V}:=\{v:[0,T]\rightarrow Q\ \ measurable\}. $$

If $u\in\mathcal{U}$, $v\in\mathcal{V}$ then denote by $x(\cdot,t_0,x_0,u,v)$ the solution of the initial value problem
$$\dot{x}(t)=f(t,x(t),u(t))+g(t,x(t),v(t)), \ \ x(t_0)=x_0. $$

We assume that the players  use control with guide strategies (CGS). In this case the control depends not only on a current position but also on a vector $w$. The vector $w$ is called a guide. The dimension of the guide can  differ from~$n$.

The control with guide strategy of player I $U$ is a  triple of functions $(u,\psi^1,\chi^1)$ such that for some natural $m$ the function
$u$ maps $[0,T]\times \mathbb{R}^n\times\mathbb{R}^m$ to $P$, the function $\psi^1$ maps $[0,T]\times [0,T]\times \mathbb{R}^n\times \mathbb{R}^m$ to $\mathbb{R}^m$, and $\chi^1$ is a function of $\ir$ with values in $\mathbb{R}^m$.

The meaning of the functions $u$, $\psi^1$, and $\chi^1$ is the following. Let  $w^1$ be a $m$-dimensional vector. Further it denotes the state of first player's guide. Player I computes the value of the variable $w^1$  using the rules which are given by the strategy $U$.
The function $u(t_*,x_*,w^1)$ is a function forming the control of player I. It depends on the current position $(t_*,x_*)$ and the current state of guide $w^1$. The function $\psi^1(t_+,t_*,x_*,w^1)$ determines the value of the guide at  time  $t_+$ under condition that  at  time  $t_*$ the phase vector is equal to $x_*$, the state of guide is equal to $w^1$. The function $\chi^1(t_0,x_0)$ determines the initial state of guide.

Player I forms his control stepwise. Let $(t_0,x_0)$ be an initial position, and let $\Delta=\{t_k\}_{k=0}^r$ be a partition of the interval $[t_0,T]$. Suppose that player II chooses his control $v[\cdot]$ arbitrarily. He can also use his own CGS and form the control $v[\cdot]$ stepwise.  Denote the solution $x[\cdot]$ of equation (\ref{system}) with the initial condition $x[t_0]=x_0$ such that the control of player I is equal to $u(t_k,x_k,w_k^1)$ on $[t_{k},t_{k+1}[$ by $x^1[\cdot,t_0,x_0,U,\Delta,v[\cdot]]$. Here the state of the system at  time   $t_k$  is  $x_k$, the state of the first player's guide is $w_k^1$; it is computed by the rule  $w_{k}^1=\psi^1(t_{k},t_{k-1},x_{k-1},w_{k-1}^1)$ for $k=\overline{1,r}$, $w_0^1=\chi^1(t_0,x_0).$

The control with guide strategy of player II is defined analogously. It is a triple  $V=(v,\psi^2,\chi^2)$. Here $v=v(t_*,x_*,w^2),$ $\psi^2=\psi^2(t_+,t_*,x_*,w^2)$, $\chi^2=\chi^2(t_0,x_0))$; $(t_*,x_*)$ is a current position, $w^2$ denotes the guide of player II, $(t_0,x_0)$ is an initial position. The motion generated by a strategy $V$, a partition $\Delta$ of the interval $[t_0,T]$, and a measurable control of player II $u[\cdot]$ is also constructed stepwise. Denote it by  $x^2[\cdot,t_*,x_*,V,\Delta,u[\cdot]]$.

We assume that the Nash equilibrium is achieved  when the players get the same partition. Let  $\Delta=\{t_k\}_{k=0}^m$ be a partition of the interval $[t_0,T]$. Denote the solution $x[\cdot]$ of equation (\ref{system}) with the initial condition $x[t_0]=x_0$ such that the control of player I is equal to $u(t_k,x_k,w_k^1)$  on $[t_k,t_{k+1}[$, and the control of player II is equal to $v(t_k,x_k,w_k^2)$ on $[t_k,t_{k+1}[$ by $x^{(c)}[\cdot,t_*,x_*,U,V,\Delta]$. Here $x_k$ denotes the state of the system at  time  $t_k$; $w_k^i$ is the state of the $i$-th player's guide at  time  $t_k$. Recall that $w_{k+1}^i=\psi^i(t_{k+1},t_k,x_k,w_k^1)$,  $w_{0}^i=\chi^i(t_0,x_0)$, $i=1,2$.

\begin{definition} Let $G\subset\ir$. A pair of control with guide strategies  $(U^{*}, V^*)$ is said to be a Control with Guide Nash equilibrium on $G$   iff for all $(t_0,x_0)\in G$ the following inequalities hold:
\begin{equation*}\begin{split}
\lim_{\delta\downarrow 0}\sup\{&\sigma_1(x^2[T,t_0,x_0,V^*,\Delta,u[\cdot]]): d(\Delta)\leq\delta,u[\cdot]\in\mathcal{U}\}\\
&\leq \lim_{\delta\downarrow 0}\inf\{\sigma_1(x^{(c)}[T,t_0,x_0,U^*,V^*,\Delta]): d(\Delta)\leq\delta\},
\end{split}
\end{equation*}
\begin{equation*}\begin{split}
\lim_{\delta\downarrow 0}\sup\{&\sigma_2(x^1[T,t_0,x_0,U^*,\Delta,v[\cdot]]): d(\Delta)\leq\delta,v[\cdot]\in\mathcal{V}\}\\
&\leq \lim_{\delta\downarrow 0}\inf\{\sigma_2(x^{(c)}[T,t_0,x_0,U^*,V^*,\Delta]): d(\Delta)\leq\delta\}.
\end{split}
\end{equation*}
\end{definition}


\section{Continuous value function}\label{sec_cont}
In this section we assume that  there exists a continuous function satisfying some viability conditions.
\subsection{Construction of Nash Equilibrium Strategies}

Let $(t_*,x_*)\in \ir$, $u_*\in P$, $v_*\in Q$.

Define
$${\rm Sol}^1(t_*,x_*;v_*):={\rm cl}\{x(\cdot,t_*,x_*,u,v_*):u\in \mathcal{U}\}, $$
$${\rm Sol}^2(t_*,x_*;u_*):={\rm cl}\{x(\cdot,t_*,x_*,u_*,v):v\in \mathcal{V}\}, $$
$${\rm Sol}(t_*,x_*):={\rm cl}\{x(\cdot,t_*,x_*,u,v):u\in \mathcal{U},v\in\mathcal{V}\}. $$
Here ${\rm cl}$ denotes the closure in the space of continuous vector function on $[0,T]$. Note, that the sets ${\rm Sol}^1(t_*,x_*;v_*)$, ${\rm Sol}^2(t_*,x_*;u_*)$, ${\rm Sol}(t_*,x_*)$ are compact.

\begin{theorem}\label{th_univ_model_cont} Let a continuous function  $(c_1,c_2):\ir\rightarrow \mathbb{R}^2$ satisfy the following conditions:
\begin{list}{\rm (F\arabic{tmp})}{\usecounter{tmp}}
\item\label{cond_boundary} $c_i(T,x)=\sigma_i(x)$, $i=1,2$;
\item\label{cond_v_stable} for every $(t_*,x_*)\in \ir$, $u\in P$ there exists a motion $y^2(\cdot)\in {\rm Sol}^2(t_*,x_*;u)$ such that $c_1(t,y^2(t))\leq c_1(t_*,x_*)$ for $t\in [t_*,T]$;
\item\label{cond_minu_stable} for every $(t_*,x_*)\in \ir$, $v\in Q$ there exists a motion $y^1(\cdot)\in {\rm Sol}^1(t_*,x_*;v)$ such that $c_2(t,y^1(t))\leq c_2(t_*,x_*)$ for $t\in [t_*,T]$;
\item\label{cond_together_stable} for every $(t_*,x_*)\in \ir$ there exists a motion $y^{(c)}(\cdot)\in {\rm Sol}(t_*,x_*)$ such that $c_i(t,y^{(c)}(t))= c_i(t_*,x_*)$ for $t\in [t_*,T]$, $i=1,2$.
\end{list}
Then for each compact $G\subset\ir$ there exists a Control with Guide Nash equilibrium on $G$. The corresponding payoff of  player $i$ is $c_i(t_0,x_0)$.
\end{theorem}

Note that  conditions (F1)--(F4) were first derived in \cite{AverboukhSb} as the sufficient condition for the function $(c_1,c_2)$ to provide Nash equilibrium payoff at the given position in the framework of Kleimenov approach. In that papers the obtained equilibria aren't universal.

The proof of Theorem \ref{th_univ_model_cont} is based on the Krasovskii-Subbotin extremal shift rule.

Let $G\subset\ir$ be a compact. Denote by $E$ the reachable set from $G$:
\begin{equation}\label{E_def}
E:= \{x(t,t_*,x_*,u,v):(t_*,x_*)\in G, t\in [t_*,T], u\in \mathcal{U},v\in\mathcal{V}\}.
\end{equation}
Put
\begin{equation}\label{K_def}
K:= \max\{\|f(t,x,u)+g(t,x,v)\|:t\in [0,T],x\in E, u\in P, v\in Q\},
\end{equation}
Let $L$ be a Lipschitz constant of the function $f+g$ on $[0,T]\times E\times P\times Q$, i.e. for all $t\in [0,T],x',x''\in E,u\in P,v\in Q$
$$\|f(t,x',u)+g(t,x',v)-f(t,x'',u)-g(t,x'',v)\|\leq L\|x'-x''\|. $$
Also, put
\begin{equation*}\begin{split}
\varphi^*(\delta):=\sup\{\|f&(t',x,u)+g(t',x,u)-f(t'',x,u)-g(t'',x,u)\|:\\
&t',t''\in [0,T], |t'-t''|\leq\delta, x\in E, u\in P,v\in Q\}.
\end{split}
\end{equation*} Note that $\varphi^*(\delta)\rightarrow 0$, as $\delta\rightarrow 0$.

Let us introduce the auxiliary  controlled system
\begin{equation}\label{shsystem}
\dot{s}=h(t,s,\omega_1,\omega_2), \ \ s\in\mathbb{R}^n, \ \ \omega_i\in \Omega_i.
\end{equation}

Below we consider two cases.
\begin{list}{(\roman{tmp})}{\usecounter{tmp}}
\item $\Omega_1=P$, $\Omega_2=Q$, $h=f+g$;
\item $\Omega_1=P\times Q$, $\Omega_2=\varnothing$, $h=f+g$.
\end{list}
Note that in  both cases  system (\ref{shsystem}) satisfies the Isaacs condition.

Put $\beta:= 2 L$, $R:=\max\{\|s'-s''\|:s',s''\in E\}$, $\varphi(\delta)=4\varphi^*(\delta)R+4K^2\delta$.

The following lemma was proved by Krasovskii and Subbotins (see \cite{NN_PDG_en})
\begin{lemma}\label{lm_krasovskii}
Let $s_1^0,s_2^0\in\mathbb{R}^n$, $t_*\in [0,T]$, $\omega_1^*\in\Omega_1,$ $\omega_2^*\in\Omega_2$ satisfy the following conditions
$$\max_{\omega_1\in \Omega_1}\min_{\omega_2\in \Omega_2}\langle s_2^0-s_1^0,h(t_*,s_1^0,\omega_1,\omega_2) \rangle=\min_{\omega_2\in \Omega_2}\langle s_2^0-s_1^0,h(t_*,s_1^0,\omega_1^*,\omega_2) \rangle, $$
$$\min_{\omega_2\in \Omega_2}\max_{\omega_1\in \Omega_1}\langle s_2^0-s_1^0,h(t_*,s_1^0,\omega_1,\omega_2) \rangle=\max_{\omega_1\in \Omega_1}\langle s_2^0-s_1^0,h(t_*,s_1^0,\omega_1,\omega_2^*) \rangle. $$

If $s_1(\cdot)$ is a solution of the initial value problem
$$\dot{s}_1=h(t,s_1,\omega_1^*,\omega_2(t)), \ \ s_1(t_*)=s_1^0, $$
and $s_2(\cdot)$ is a solution of the initial value problem
$$\dot{s}_2=h(t,s_2,\omega_1(t),\omega_2^*), \ \ s_2(t_*)=s_2^0, $$ for some measurable controls $\omega_1(\cdot)$ and $\omega_2(\cdot)$, then  for all $t_+\in [t_*,T]$ the estimate
$$\|s_2(t_+)-s_1(t_+)\|^2\leq \|s_2^0-s_1^0\|^2(1+\beta(t_+-t_*))+\varphi(t_+-t_*)\cdot(t_+-t_*)$$ is fulfilled.
\end{lemma}

We assume that the $i$-th player's  guide $w^i$ is a quadruple $(d^i,\tau^i,w^{i,(a)},w^{i,(c)})$. The variable $d^i\in\mathbb{R}$ describes an accumulated error, $\tau^i\in [0,T]$ is a previous time  of the control correction, $w^{i,(a)}\in\mathbb{R}^n$ is a punishment part of the guide, and $w^{i,(c)}\in\mathbb{R}^n$ is a consistent part of the guide. The whole dimension of the guide is $2n+2$.

For any  $(t_*,x_*)\in\ir$, $u\in P$, $v\in Q$ choose and fix a motion $y^2(\cdot;t_*,x_*,u)$ satisfying  condition (F\ref{cond_v_stable}), a motion $y^1(\cdot;t_*,x_*,v)$ satisfying condition (F\ref{cond_minu_stable}), and  a motion  $y^{(c)}(\cdot;t_*,x_*)$ satisfying condition (F\ref{cond_together_stable}).

Now let us define the strategies $U^*$ and $V^*$. Below we prove that the pair of strategies $(U^*,V^*)$ is a  Control with Guide Nash equilibrium on $G$.

First put $\chi^1(t_0,x_0)=\chi^2(t_0,x_0):= (0,t_0,x_0,x_0)$.

Let  $(t,x)$ be a position, $w^i=(d^i,\tau^i,w^{i,(a)},w^{i,(c)})$ be a state of the $i$-th player's guide. Put
\begin{equation}\label{z_def}
z^i:=\left\{
\begin{array}{cl}
  w^{i,(c)}, &  \|w^{i,(c)}-x\|^2\leq d^i(1+\beta(t-\tau^i))+\varphi(t-\tau^i)(t-\tau^i), \\
  w^{i,(a)}, & \mbox{ otherwise}.
\end{array}\right.
\end{equation}

Let us consider two cases.

\begin{list}{$i=$\arabic{tmp}.}{\usecounter{tmp}}
  \item  Choose a control $u_*$ by the rule
\begin{equation}\label{u_star_choice}
\max_{u\in P}\langle z^1-x,f(t,x,u)\rangle=\langle z^1-x,f(t,x,u_*)\rangle.
\end{equation} Further, let $v^*$ satisfy the following condition
\begin{equation}\label{v_up_star_choice}
\min_{v\in Q}\langle z^1-x,g(t,x,v)\rangle=\langle z^1-x,g(t,x,v^*)\rangle.
\end{equation} Define $u(t,x,w^1):= u_*$. For $t_+>t$ put $\psi^1(t_+,t,x,w^1)$ be equal to $w^1_+=(d_+^1,\tau_+^1,w^{1,(a)}_+,w^{1,(c)}_+)$, where
$$d_+^1:= \|z^1-x\|^2,\ \ \tau_+^1:=t,\ \ w^{1,(a)}_+:= y^1(t_+;t,z^1,v^*),\ \ w^{1,(c)}_+:=y^{(c)}(t_+;t,z^1).$$
  \item Let a control $v_*$ be such that
\begin{equation}\label{v_star_choice}
\max_{v\in Q}\langle z^2-x,g(t,x,v)\rangle=\langle z^2-x,g(t,x,v_*)\rangle.
\end{equation} Choose $u^*$ satisfying the  condition
\begin{equation}\label{u_up_star_choice}
\min_{u\in P}\langle z^2-x,f(t,x,u)\rangle=\langle z^2-x,f(t,x,u^*)\rangle.
\end{equation}
Set $v(t,x,w):= v_*$. For $t_+>t$ put $\psi^2(t_+,t,x,w^2)$ be equal to $w_+^2=(d_+^2,\tau_+^2,w^{2,(a)}_+,w^{2,(c)}_+)$, where $$d_+^2:= \|z^2-x\|^2,\ \ \tau_+^2:= t,\ \  w^{2,(a)}_+:= y^2(t_+;t,z^2,u^*),\ \  w^{2,(c)}_+:= y^{(c)}(t_+;t,z^2).$$
\end{list}

Note that
\begin{equation}\label{c_equality}
    c_j(t_+,w^{i,(c)}_+)=c_j(t,z^i) \ \ \mbox{for all } i,j=1,2,
\end{equation}
\begin{equation}\label{c_ineq}
    c_1(t_+,w^{2,(a)}_+)\leq c_1(t,z^2), \ \ c_2(t_+,w^{1,(a)}_+)\leq c_2(t,z^1).
\end{equation}

Below let $x_+$ denote the state of the system at  time  $t_+$. 

\begin{lemma}\label{lm_x_w_dist}
Suppose that  $z^1=z^2=z$.  If players I and II use respectively the controls  $u_*$ and $v_*$   on the interval  $[t,t_+]$, then $w^{1,(c)}_+=w^{2,(c)}_+$ and
$$\|x_+-w^{i,(c)}_+\|^2\leq d_+^i(1+\beta(t_+-\tau_+))+\varphi(t_+-\tau_+)(t_+-\tau_+). $$
\end{lemma}
\begin{proof}
The controls $u_*$ and $v_*$ satisfy the  condition
$$\max_{u\in P,v\in Q}\langle z-x,f(t,x,u)+g(t,x,v)\rangle=\langle z-x,f(t,x,u_*)+g(t,x,v_*)\rangle. $$ We apply Lemma \ref{lm_krasovskii} with $\Omega_{1}=P\times Q$, $\Omega_2=\varnothing$, $h=f+g$. If $x(\cdot)=x(\cdot,t,x,u_*,v_*)$, $y^{(c)}(\cdot)=y^{(c)}(\cdot;t,z)$, then
$$
\|x(t_+)-y^{(c)}(t_+)\|^2\leq \|x-z\|^2(1+\beta(t_+-t))+\varphi(t_+-t)\cdot (t_+-t).
$$
The definition of the strategies $U^*$ and $V^*$ yields that $w_+^{i,(c)}=y^{(c)}(t_+)$ for $i=1,2$. By construction of the functions $\psi_i$, $i=1,2$ we have that $t=\tau^i_+$, and $d_+^i=\|x-z\|^2$. This completes the proof of the Lemma.
\end{proof}

\begin{lemma}\label{lm_x_zero_dis}
If player I uses the control $u_*$ on the interval  $[t,t_+]$, then
$$\|x_+-w^{1,(a)}_+\|^2\leq d_+^i(1+\beta(t_+-\tau_+))+\varphi(t_+-\tau_+)(t_+-\tau_+),\ \  i=1,2. $$
\end{lemma}
\begin{proof}
We apply Lemma \ref{lm_krasovskii} with $\Omega_1=P$, $\Omega_2=Q$ and $h=f+g$. The choice of  $u_*$ (see (\ref{u_star_choice})) and $v^*$ (see (\ref{v_up_star_choice})) yields that the  inequality
$$
\|x(t_+)-y^1(t_+)\|^2\leq \|x-z^1\|^2(1+\beta(t_+-t))+\varphi(t_+-t)\cdot (t_+-t)
$$ holds with $x(\cdot)=x(\cdot,t,x,u_*,v)$ and $y^1(\cdot)=y^1(\cdot,t,z^1,v^*)$. Since $w^{1,(a)}_+=y^1(t_+)$, $\tau^1_+=t$, and $d_+^1=\|x-z^1\|^2$, the conclusion of the Lemma follows.
\end{proof}

We need the following estimate. Let   $\Delta=\{t_k\}_{k=0}^r$ be a partition of the interval  $[t_0,T]$, and let $\{\gamma_k\}_{k=0}^r$ be a collection of numbers such that
\begin{equation}\label{gamma_step}
\gamma_{k+1}\leq \gamma_k(1+\beta(t_{k+1}-t_k))+\varphi((t_{k+1}-t_k))\cdot ((t_{k+1}-t_k)).
\end{equation}
Then,
\begin{equation}\label{kras_subb_lemma}
    \gamma_{k}\leq [\gamma_0+(1+(t_k-t_0))\varphi(d(\Delta))]\exp\beta(t_k-t_0).
\end{equation}

\vspace{2pt}
\noindent\textit{Proof of Theorem \ref{th_univ_model_cont}}
First let us show that for all $(t_0,x_0)\in G$ the following equality is valid:
\begin{equation}\label{sigma_c}
    c_j(t_0,x_0)=\lim_{\delta\downarrow 0}\inf\{\sigma_j(x^{(c)}[T,t_0,x_0,U^*,V^*,\Delta]), d(\Delta)\leq\delta\}, \ \ j=1,2.
\end{equation}

Let   $\Delta=\{t_k\}_{k=1}^r$ be a partition of the interval $[t_0,T]$. 
Denote the state of the system at time $t_k$ by $x_k$, the state of the $i$-th player's  guide by $w^i_k=(d_k^i,\tau_k,w^{(a),i}_k,w^{i,(c)}_k)$. Also let  $z^i_k$ be chosen by rule (\ref{z_def}) at  time   $t_k$. We have that  $\tau_0=t_0$, $\tau_{k+1}=t_k$ for $k\geq 0$. Moreover, $$z^1_0=w^{1,(c)}_0=w^{1,(c)}_0=z^2_0.$$ Hence using  lemma  \ref{lm_x_w_dist} inductively we get that
 \begin{equation}\label{z_w_step_eq}
 z^1_k=w^{1,(c)}_k=z^2_k=w^{2,(c)}_k,\ \ d^i_{k+1}=\|x_k-z_k^i\|^2.
\end{equation}
and
$$\|x_{k+1}-z^i_{k+1}\|^2\leq \|x_{k}-z^i_{k}\|^2(1+\beta(t_{k+1}-t_k))+\varphi(t_{k+1}-t_k)(t_{k+1}-t_k) $$ for all $k=\overline{0,N}$.

It follows from (\ref{kras_subb_lemma}) that
$$
    \|x_r-z_r^i\|^2\leq [\|x_0-z_0^i\|^2+(1+(t_r-t_0))\varphi(d(\Delta))]\exp\beta(t_r-t_0).
$$
Since $z_0^i=x_0$, we obtain that
\begin{equation}\label{x_z_estima}
\|x_r-z_r^i\|\leq \varkappa(\delta):= \Bigl[(1+(t_r-t_0))\varphi(\delta)\exp\beta(t_r-t_0)\Bigr]^{1/2},
\end{equation} where $\delta=d(\Delta)$. Note that $\varkappa(\delta)\rightarrow 0$, $\delta\rightarrow 0$.

Let $\phi_j( Y)$ be a modulus of continuity  of the function $\sigma_j$ on the set $E$
$$\phi_j(\gamma):=\sup\{|\sigma_j(x')-\sigma_j(x'')|:x',x''\in E, \|x'-x''\|\leq\gamma\}.$$  We have, that
\begin{equation}\label{sigma_cont_mod}
\|\sigma_j(x_r)-\sigma_j(z_r^i)\|\leq \phi_j(\varkappa(\delta)).
\end{equation}

Since $z^i_k=w^{i,(c)}_k$, it follows from  (\ref{c_equality}) that
$c_j(t_{k+1},w_{k+1}^{i,(c)})=c_j(t_k,z^i_k)=c_j(t_{k},w_{k}^{i,(c)})$. Therefore, using condition (F\ref{cond_boundary}) we get
$$\|\sigma_j(x[T,t_0,x_0,U^*,V^*,\Delta])-c_j(t_0,x_0)\|\leq \phi_j(\varkappa(\delta)) $$ with $\delta=d(\Delta)$. Passing to the limit we obtain equality  (\ref{sigma_c}).

Now let us show that for all $(t_0,x_0)\in G$
\begin{equation}\label{sigma_c_deviate}
     c_2(t_0,x_0)\geq \lim_{\delta\downarrow 0}\sup\{\sigma_2(x^1[T,t_0,x_0,U^*,\Delta,v[\cdot]), d(\Delta)\leq\delta,v[\cdot]\in\mathcal{V}\}.
\end{equation}

Let    $\Delta=\{t_k\}_{k=1}^r$ be a partition of the interval $[t_0,T]$, and let $v[\cdot]$ be a control of player II.
Denote the state of the system at time $t_k$ by $x_k$, the state of the first player's guide by $w^1_k=(d_k^1,\tau_k,w^{(a),1}_k,w^{1,(c)}_k)$. Also let  $z^1_k$ be chosen by  rule (\ref{z_def}) at  time   $t_k$.

We claim that inequality (\ref{gamma_step}) is valid with $\gamma_k=\|z_{k}^1-x_k\|^2$.
Note that $\tau^1_{k+1}=t_k$, $d^1_{k+1}=\|z^1_k-x_k\|^2$. If $z^1_{k+1}=w^{1,(c)}_{k+1}$, then inequality (\ref{gamma_step}) holds by construction. If $z^1_{k+1}=w^{1,(c)}_{k+1}$, then using lemma \ref{lm_x_zero_dis} we obtain that inequality (\ref{gamma_step}) is fulfilled also.

Therefore, we have inequality  (\ref{kras_subb_lemma}) with  $\gamma_0=0$ and $\gamma_k=\|z_{k}^1-x_k\|^2$. Hence, $$
\|z^1_r-x_r\|\leq \varkappa(d(\Delta)).
$$
Consequently, inequality (\ref{sigma_cont_mod}) is fulfilled for $i=1$, $j=2$.

It follows from (\ref{z_def}), (\ref{c_equality}), and (\ref{c_ineq})  that
\begin{equation}\label{c_z_k_plus_k}
    c_2(t_{k+1},z^1_{k+1})\leq c_2(t_k,z^1_k).
\end{equation}

Condition  (F\ref{cond_boundary}) and the equality  $z_0^1=x_0$ yield the inequality
$$
    \sigma_2(z_r^1)=c_2(T,z_r^1)\leq c_2(t_0,x_0).
$$
From this and (\ref{c_z_k_plus_k}) we conclude that
$$\sigma_2(x^1[T,t_0,x_0,U^*,\Delta,v[\cdot]])\leq c_2(t_0,x_0)+\phi_2(\varkappa(\delta)), $$ with $\delta=d(\Delta)$.
Passing to the limit we get inequality (\ref{sigma_c_deviate}).

Analogously one can prove the inequality
\begin{equation}\label{c_2_deviate}
     c_1(t_0,x_0)\geq \lim_{\delta\downarrow 0}\sup\{\sigma_1(x^2[T,t_0,x_0,V^*,\Delta,u[\cdot]), d(\Delta)\leq\delta,u[\cdot]\in\mathcal{U}\}.
\end{equation} Combining equality (\ref{sigma_c}) and inequalities (\ref{sigma_c_deviate}), (\ref{c_2_deviate}) we conclude that the strategies $U^*$ and $V^*$ form the Control with Guide Nash equilibrium on $G$. Moreover the Nash equilibrium payoff of player $i$ at the position $(t_0,x_0)$ is  $c_i(t_0,x_0)$.

\qed


\subsection{Infinitesimal Form of Conditions (F1)--(F4) }

Define  $$H_1(t,x,s):= \max_{u\in P}\min_{v\in Q}\langle s,f(t,x,u)+g(t,x,v)\rangle,$$ $$H_2(t,x,s):= \max_{v\in Q}\min_{u\in P}\langle s,f(t,x,u)+g(t,x,v)\rangle. $$

\begin{proposition}
Conditions (F\ref{cond_v_stable}), and (F\ref{cond_minu_stable}) are equivalent  to the the following one: the function $c_i$ is viscosity supersolution of the equation \begin{equation}\label{HJ_aux}\frac{\partial c_i}{\partial t}+H_i(t,x,\nabla c_i)=0. \end{equation}
\end{proposition}
This Proposition directly follows from \cite[Theorem 6.4]{Subb_book}.

Further, define a modulus derivative at the position $(t,x)$ in the direction $w\in\mathbb{R}^n$ by the rule
\begin{equation*}
{\rm d}_{abs}(c_1,c_2)(t,x;w)
:=\liminf_{\delta\downarrow 0,w'\rightarrow w}\frac{|c_1(t+\delta,x+\delta w')-c_1(t,x)|+|c_2(t+\delta,x+\delta w')-c_2(t,x)|}{\delta}.
\end{equation*}

\begin{proposition}
Condition (F\ref{cond_together_stable}) is valid if and only if for every ${(t,x)\in \ir}$ $$\inf_{w\in\mathcal{F}(t,x)}{\rm d}_{abs}(c_1,c_2)(t,x;w)=0. $$
\end{proposition}
\begin{proof}
Condition (F4) means that the graph of the function $(c_1,c_2)$ is viable under the differential inclusion
$$
\left(
\begin{array}{c}
  \dot{x} \\
  \dot{J}_1 \\
  \dot{J}_2
\end{array}
\right)=
{\rm co}\left\{\left(
\begin{array}{c}
  f(t,x,u)+g(t,x,v) \\
  0 \\
  0
\end{array}
\right):u\in P,v\in Q\right\}.
$$
One can rewrite this condition in the infinitesimal form \cite[Theorem 11.1.3]{Viability}: for $J_1=c_1(t,x)$, $J_2=c_2(t,x)$ and some $w\in {\rm co}\{f(t,x,u)+g(t,x,v):u\in P,v\in Q\} $ the inclusion
\begin{equation}\label{weak_invar_D}
\left(
\begin{array}{c}
  w \\
  0 \\
  0
\end{array}
\right)\in D{\rm gr}(c_1,c_2)(t,(x,J_1,J_2))
\end{equation}
 holds.
Here $D$ denotes the contingent derivative. It is defined in the following way. Let $\mathcal{G}\subset [0,T]\times \mathbb{R}^m$, $\mathcal{G}[t]$ denote a section of $\mathcal{G}$ by $t$: $$\mathcal{G}[t]:= \{w\in \mathbb{R}^m:(t,x)\in \mathcal{G}\},$$ and let the symbol ${\rm d}$ denote the Euclidian distance between a point and a set. Following \cite{Viability} set $$D\mathcal{G}(t,y):= \left\{h\in\mathbb{R}^m: \liminf_{\delta\rightarrow 0}\frac{{\rm d}(y+\delta h; \mathcal{G}[t+\delta])}{\delta}=0\right\}. $$

Let $J_i=c_i(t,x)$. We have that $(w, Y_1, Y_2)\in D{\rm gr}(c_1,c_2)(t,(x,J_1,J_2)) $   if and only if there exist sequences $\{w_k\}_{k=1}^\infty$ and $\{\delta_k\}_{k=1}^\infty$ such that $w=\lim_{k\rightarrow\infty}w_k$, and $$ Y_i=\lim_{k\rightarrow\infty}\frac{c_i(t+\delta_k,x+\delta_kw_k)-c_i(t,x)}{\delta_k}.$$
Therefore, condition (\ref{weak_invar_D}) is equivalent to the condition $d_{abs}(c_1,c_2)(t,x;w)=0$ for some $w\in {\rm co}\{f(t,x,u)+g(t,x,v):u\in P,v\in Q\}$.
\end{proof}

\subsection{System of Hamilton--Jacobi equations}
Let us show that Theorem \ref{th_univ_model_cont} generalizes the method based on the system of Hamilton--Jacobi equations.

It is well known that the solutions of the system of Hamilton--Jacobi equations provide  Nash equilibria \cite{Basar}.

For any $s\in \mathbb{R}^n$ let $\hat{u}(t,x,s_1)$ satisfy the condition
$$\langle s,f(t,x,\hat{u}(t,x,s))\rangle=\max\{\langle s,f(t,x,u)\rangle:u\in P\}, $$ and let $\hat{v}(t,x,s)$ satisfy the condition
$$\langle s,g(t,x,\hat{v}(t,x,s))\rangle=\max\{\langle s,g(t,x,u)\rangle:u\in P\}. $$
Set
$$\mathcal{H}_i(t,x,s_1,s_2):= \langle s_i, f(t,x,\hat{u}(t,x,s_1))+g(t,x,\hat{v}(t,x,s_2))\rangle. $$

Consider the system of Hamilton--Jacobi equations
\begin{equation}\label{HJsystem}
    \left\{
    \begin{array}{l}
      \frac{\partial \varphi_i}{\partial t}+\mathcal{H}_i(t,x,\nabla \varphi_1,\nabla \varphi_2)  =0,  \\
      \varphi_i(T,x)=\sigma_i(x).
    \end{array}
    \right. \ \ i=1,2
\end{equation}

\begin{proposition}
If the function $(\varphi_1,\varphi_2)$ is a classical solution of  system (\ref{HJsystem}),  then it satisfies condition (F1)--(F4). \end{proposition}
\begin{proof}
Condition (F1) is obvious.

Since $(\varphi_1,\varphi_2)$ is the solution of system (\ref{HJsystem}), we have that
\begin{equation*}
\begin{split} 
 0=&\frac{\partial \varphi_1(t,x)}{\partial t}+\max_{u\in P}\langle \nabla \varphi_1(t,x),f(t,x,u)\rangle
+\langle \nabla \varphi_1(t,x),g(t,x,\hat{v}(t,x,\nabla \varphi_1(t,x)))\rangle\\
\geq  &\frac{\partial \varphi_1(t,x)}{\partial t}+\max_{u\in P}\langle \nabla \varphi_1(t,x),f(t,x,u)\rangle
+\min_{v\in Q}\langle \nabla \varphi_1(t,x),g(t,x,v)\rangle\\
=&\frac{\partial \varphi_1(t,x)}{\partial t}+H_1(t,x,\nabla \varphi_1(t,x)).
\end{split}
\end{equation*} 

The subdifferential of the smooth function $\varphi_1$ is equal to $D^-\varphi_1(t,x)=\{(\partial \varphi_1(t,x)/\partial t,\nabla \varphi_1(t,x))\}$. Therefore, $\varphi_1$ is a viscosity supersolution of equation (\ref{HJ_aux}) for $i=1$ \cite[Definition (U4)]{Subb_book}. This is equivalent to condition (F2).

Condition (F3) is proved in the same way.
\begin{equation*}
d_{abs}(\varphi_1,\varphi_2)(t,x;w)=\left|\frac{\partial \varphi_1(t,x)}{\partial t}+\langle \nabla \varphi_1(t,x),w\rangle\right|+\left|\frac{\partial \varphi_2(t,x)}{\partial t}+\langle \nabla \varphi_2(t,x),w\rangle\right|.
\end{equation*}

Substituting $w=f(t,x,\hat{u}(t,x,\nabla \varphi_1(t,x)))+g(t,x,\hat{v}(t,x,\nabla \varphi_2(t,x)))$ gives condition~(F4).
\end{proof}

Generally, there exists a smooth function $(c_1,c_2)$ satisfying conditions (F1)--(F4)  not being a solution of the system of Hamilton--Jacobi equations.

\begin{example}
 Consider the  system
\begin{equation}\label{sys_ex_dis}
    \left\{\begin{array}{cc}
      \dot{x}_1= & -v \\
      \dot{x}_2= & 2u+v
    \end{array}\right.
\end{equation}
Here $t\in [0,1]$, $u,v\in [-1,1]$. The purpose of the $i$-th player is to maximize  $x_i(1)$.

The function $(c_1^*,c_2^*)$ with $c_1^*(t,x_1,x_2)=x_1+(1-t)$, $c_2^*(t,x_1,x_2)=x_2+(1-t)$ satisfies conditions (F1)--(F4), but it is not a solution of the system of Hamilton--Jacobi equations (\ref{HJsystem}). Moreover, $c_i^*(t,x)> \varphi_i(t,x)$ for some solutions of system (\ref{HJsystem}) $(\varphi_1,\varphi_2)$.
\end{example}
\begin{proof}
First let us write down the system of Hamilton--Jacobi equations for the case under consideration. Denote $\partial \varphi_1/\partial x_j$ by $p_j$, $\partial \varphi_2/\partial x_j$ by $q_j$.

The variables  $\hat{u}$ and $\hat{v}$ satisfy the conditions
$$\max_{u\in [-1,1]}p_2u=p_2\hat{u}, \ \ \max_{v\in [-1,1]}(-q_1+q_2)v=(-q_1+q_2)\hat{v}.$$

Hence the system of Hamilton--Jacobi equations (\ref{HJsystem}) takes the form
\begin{equation}\label{HJ_example}
 \left\{
\begin{array}{cc}
  \frac{\partial \varphi_1}{\partial t}- p_1 \hat{v}+p_2(2\hat{u}+\hat{v}) & =0 \vspace{4pt},\\

  \frac{\partial \varphi_2}{\partial t}-q_1 \hat{v}+q_2(2\hat{u}+\hat{v}) & =0.
\end{array}\right.
\end{equation}

The boundary conditions are $\varphi_1(1,x_1,x_2)=x_1$, $\varphi_2(1,x_1,x_2)=x_2$.

The function $(c_1^*,c_2^*)$ satisfies conditions (F1)--(F4). Indeed, condition (F1) holds obviously. Condition (F2)  is valid with  $v=1$, analogously condition  (F3) is valid with  $u=-1$. Moreover both players can keep the values of the functions if they use the controls $v=-1$, $u=1$. This means that condition (F4) holds also.

On the other hand the pair of functions $(c_1^*,c_2^*)$ does not satisfy the system of Hamilton--Jacobi equations. Indeed,
\begin{equation*}\begin{split}
&\partial c_1^*/\partial x_1=p_1=1, \ \ \partial c_1^*/\partial x_2=p_2=0, \ \ \partial c_2^*/\partial x_1=q_1=0, \\
&\partial c_2^*/\partial x_2=q_2=1, \ \ \partial c_1^*/\partial t=\partial c_2^*/\partial t=-1.\end{split}
\end{equation*}
Therefore, $\hat{v}=1$. Substitution into the first equation of (\ref{HJ_example}) leads to the contradiction.

Further, consider the functions $\varphi_1(t,x_1,x_2)=x_1-(1-t)$, $\varphi_2^\alpha(t,x_1,x_2)=x_2+(1+2\alpha)(1-t)$. Here $\alpha$ is a parameter from  $[-1,1]$. Note that if $\hat{v}=1$ and  $\hat{u}=\alpha$, then $(\varphi_1,\varphi_2^\alpha)$ is a classical solution of system (\ref{HJ_example}).

We have that for $\alpha\in [-1,0)$
$$c_1^*(t,x_1,x_2)> \varphi_1(t,x_1,x_2), \ \ c_1^*(t,x_1,x_2)> \varphi_1^\alpha(t,x_1,x_2).$$
\end{proof}

\subsection{Problem of Continuous Value Function Existence}
The continuous function $(c_1,c_2)$ satisfying conditions (F\ref{cond_boundary})--(F\ref{cond_together_stable}) does not exist in the general case.

\begin{example}
Let the dynamics of the system be given by $$\dot{x}=u, \ \ t\in [0,1], x\in\mathbb{R}, u\in [-1,1].$$ The purpose of the first player is to maximize  $|x(1)|$. The second player is fictitious, and his purpose is to maximize  $x(1)$. In this case there is no continuous function satisfying conditions (F1)--(F4).
\end{example}
\begin{proof}
Let a function $(c_1,c_2):[0,1]\times\mathbb{R}\rightarrow \mathbb{R}^2$ satisfy conditions (F\ref{cond_boundary})--(F\ref{cond_together_stable}). Condition (F\ref{cond_v_stable}) means that
\begin{equation}\label{c_1_upper_ex}
c_1(t,x)\geq c_1\left(t_+,x+\int_{t}^{t_+}u(\theta)d\theta\right)
\end{equation}
 for any $u\in\mathcal{U}$, $t_+\in [t,1]$. In particular,
$c_1(t,x)\geq |x|+(1-t)$. Condition  (F\ref{cond_together_stable}) means that there exists a control  $u_*$ such that
\begin{equation}\label{c_i_example}
c_1(t,x)=\left|x+\int_{t}^1u_*(\tau)d\tau\right|, \ \ c_2(t,x)=x+\int_{t}^1u_*(\tau)d\tau.
\end{equation}
This yields the inequality  $$c_1(t,x)\leq \max_{u\in [-1,1]}|x+u(1-t)|=|x|+(1-t).$$ From this, and (\ref{c_1_upper_ex}) it follows that $c_1(t,x)=|x|+(1-t)$. Moreover, $u_*(\cdot)\equiv 1$ for $x\geq 0$, and $u_*(\cdot)\equiv -1$ for $x\leq 0$. Hence, $c_2(t,x)=x+(1-t)$ for $x>0$ and $c_2(t,x)=x-(1-t)$ for $x<0$.
\end{proof}

The example shows that we need to modify Theorem \ref{th_univ_model_cont} for the case of discontinuous value functions.

\section{Multivalued Value Functions}
\begin{theorem}\label{th_univ_model_mani} Assume that there exists  an upper semicontinuous multivalued function ${S:\ir\rightrightarrows \mathbb{R}^2}$ with nonempty images satisfying the following conditions:
\begin{list}{\rm (S\arabic{tmp})}{\usecounter{tmp}}
\item\label{cond_boundary_m} $S(T,x)=\{(\sigma_1(x),\sigma_2(x))\}$, $x\in\mathbb{R}^n$;
\item\label{cond_v_stable_m} for all $(t,x)\in \ir$, $(J_1,J_2)\in S(t,x)$, $u\in P$ and $t_+\in [t,T]$ there exist a motion $y^2(\cdot)\in {\rm Sol}^2(t,x;u)$ and a pair $(J_1',J_2')\in S(t_+,y^2(t_+))$ such that $J_1\geq J_1'$;
\item\label{cond_minu_stable_m} for all $(t,x)\in \ir$, $(J_1,J_2)\in S(t,x)$, $v\in Q$ and $t_+\in [t,T]$ there exist a motion $y^1(\cdot)\in {\rm Sol}^1(t,x;u)$ and a pair $(J_1'',J_2'')\in S(t_+,y^1(t_+))$ such that $J_2\geq J_2''$;
\item\label{cond_together_stable_m} for all $(t,x)\in \ir$, $(J_1,J_2)\in S(t,x)$ and $t_+\in [t,T]$ there exists a motion $y^{(c)}(\cdot)\in {\rm Sol}(t_*,x_*)$ such that $(J_1,J_2)\in S(t_+,y^{(c)}(t_+))$.
\end{list}
Then for any selector $(\hat{J}_1,\hat{J}_2)$ of the multivalued function $S$ and a compact set $G\subset\ir$ there exists a Control with Guide Nash equilibrium on $G$ such that corresponding Nash equilibrium payoff at $(t_0,x_0)\in G$ is $(\hat{J}_1(t_0,x_0),\hat{J}_2(t_0,x_0))\in S(t_0,x_0)$.
\end{theorem}
\begin{remark}Let $U_*$, $V_*$ be Nash equilibrium strategies constructed for the  compact $G\subset\ir$ and the selector $(\hat{J}_1,\hat{J}_2)$. The value of $(\hat{J}_1,\hat{J}_2)$ may vary along the Nash trajectory $x_*^c[\cdot]$, that is a limit of step-by-step  motions generated by $U_*$ and $V_*$. However, it follows from Theorem \ref{th_univ_model_mani} that for any intermediate time instant $\theta$ there exists a pair of Nash equilibrium strategies such that the corresponding Nash equilibrium payoff at $(\theta,x^c_*[\theta])$ is equal to the value of $(\hat{J}_1,\hat{J}_2)$ at the initial position.

Analogously, if $x_*^1[\cdot]$ is a limit of step-by-by step motions generated by strategy of player I $U_*$, and a control of player II $v[\cdot]$, then for any intermediate time instant $\theta$ there exists a pair of Nash equilibrium strategies such that the corresponding Nash equilibrium payoff at $(\theta,x^1_*[\theta])$ of the player II doesn't exceed the value of the function $\hat{J}_2$ at the initial position.
\end{remark}
\vspace{2pt}
\noindent\textit{Proof of Theorem \ref{th_univ_model_mani}}
To prove the theorem we modify the construction of the guide proposed in the proof of Theorem \ref{th_univ_model_cont}. We assume that the guide consists of the following components: $d\in\mathbb{R}$ is an accumulated error, $\tau\in \mathbb{R}$ is a previous time  of correction, $w^{(a)}$ is a punishment part of the guide, $w^{(c)}$ is a consistent part of the guide,  $ Y_1\in \mathbb{R}$, $ Y_2\in \mathbb{R}$ are expected payoffs of the players.

Let $(t,x)\in\ir$ be a position,  $t_+>t$, $(J_1,J_2)\in S(t,x)$, $u\in P$, $v\in Q$. Let a motion $y^2(\cdot)$ satisfy condition (S\ref{cond_v_stable_m}). Denote $b^2(t_+,t,x,J_1,J_2,u):= y^2(t_+)$. Analogously let $y^1(\cdot)$ satisfy condition (S\ref{cond_minu_stable_m}). Put $b^1(t_+,t,x,J_1,J_2,v):= y^1(t_+)$. Also, if $y^{(c)}(\cdot)$ satisfies condition (S\ref{cond_together_stable_m}), then denote $b^c(t_+,t,x,J_1,J_2):= y^{(c)}(t_+).$ 

First let us define the functions   $$\chi_1(t,x)=\chi_2(t,x):=(d_0,\tau_0,w^{(c)}_0,w^{(a)}_0, Y_{1,0}, Y_{2,0})$$ by the following rule: $d_0:= 0$, $\tau_0:= t$, $w^{(c)}_0=w^{(a)}_0:= x$, $ Y_{1,0}:= \hat{J}_1(t_0,x_0)$, $ Y_{2,0}:= \hat{J}_2(t_0,x_0)$.

Now we shall define controls and transitional functions of the guides.  Let   $t$ be a time instant. Assume that at time  $t$ the state of the system is  $x$, and the state of the $i$-th player's guide is  $w^i=(d^i,\tau^i,w^{(a),i},w^{(c),i}, Y_1^i, Y_2^i)$. Define $z^i$ by rule (\ref{z_def}).
Now let us consider the case of the first player. Put
$$
( Y_{1,+}^1, Y_{2,+}^1):=\left\{
\begin{array}{cc}
  ( Y_1^i, Y_2^i), & z^1=w^{(c),1} \\
  ( Y_1'', Y_2''), & z^1=w^{(a),1}.
\end{array}\right.
$$
 Here $( Y_1'', Y_2'')$ is an element of  $S(t,w^{(a),1})$ such that $ Y_2''=\min\{J_2:(J_1,J_2)\in S(t,w^{(a),1})\}$.
Choose $u_*$ by rule (\ref{u_star_choice}), and $v^*$ by (\ref{v_up_star_choice}). As above, put $u(t,x,w):= u_*$, also set $\psi_1(t_+,t,x,w^1):= (d^1_+,\tau^1_+,w^{(a),1}_+,w^{(c),1}_+, Y_{1,+}^1, Y_{2,+}^1)$ where
\begin{equation*}
\begin{split}
d_+^1=\|z^1-x\|^2, \ \  \tau_+^1=t, \ \  &w^{(a),1}_+=b_1(t_+,t,z^1, Y_{1,+}^1, Y_{2,+}^1,v_*),\\  &w^{(c),1}_+=b_c(t_+,t,z^1, Y_{1,+}^1, Y_{2,+}^1).
\end{split}
\end{equation*}

The case of the second player is considered in the same way.
Put
$$( Y_{1,+}^2, Y_{2,+}^2):=\left\{
\begin{array}{cc}
  ( Y_1^i, Y_2^i), & z^2=w^{(c),2} \\
  ( Y_1', Y_2'), & z^2=w^{(a),2}.
\end{array}\right.
$$ Here  $( Y_1', Y_2')$ is an element of $S(t_+,w^{(a),2})$ such that $ Y_1'=\min\{J_1:(J_1,J_2)\in S(t,w^{(a),2})\}$.
Let $v_*$ satisfy condition (\ref{v_star_choice}). Also, let $u^*$ satisfy condition (\ref{u_up_star_choice}). Put $v(t,x,w):= v_*$. Further, set $\psi_2(t_+,t,x,w^2):= (d^2_+,\tau^2_+,w^{(a),1}_+,w^{(c),2}_+, Y_{1,+}^2, Y_{2,+}^2)$ where
\begin{equation*}
\begin{split}
d_+^2=\|z^2-x\|^2, \ \ \tau_+=t,\ \  &w^{(a),2}_+=b_2(t_+,t,z^2, Y_{1,+}^2, Y_{2,+}^2,v_*),\\  &w^{(c),2}_+=b_c(t_+,t,z^2, Y_{1,+}^2, Y_{2,+}^2).
\end{split}
\end{equation*}

Let us show that for any position  $(t_0,x_0)\in G$ the following equality is fulfilled
\begin{equation}\label{J_hat}
    \hat{J}_i=\lim_{\delta\downarrow 0}\inf\{\sigma_i(x^{(c)}[T,t_0,x_0,U^*,V^*,\Delta]), d(\Delta)\leq\delta\}, \ \ i=1,2.
\end{equation}

Let $\Delta=\{t_k\}_{k=0}^r$ be a partition of  $[t_0,T]$, $d(\Delta)\leq\delta$, $x^c[\cdot]:= x^c[\cdot,t_*,x_*,U^*,V^*,\Delta]$. Extend the partition  $\Delta$ by adding the element   $t_{r+1}=t_r=T$. Denote $x_k:= x^c[t_k]$. Let us denote the state of the $i$-th player's guide at time  $t_k$  by $w^i_k=(d^i_k,w^{(a),i}_k,w^{(c),i}_k, Y_{1,k}^i, Y_{2,k}^i)$. Let  $z^i_k$ be a position chosen by rule (\ref{z_def}) for the $i$-th  player at time  $t_k$.

It follows from lemma \ref{lm_x_w_dist} that the point $z^i_k$ is equal to $w^{(c),i}_k$. In addition, $w_k^{(c),1}=w_k^{(c),2}$, and the following inequality is valid:
  $$
  \|x_k-w^{(c),i}_k\|\leq \|x_{k-1}-z_{k-1}^i\|^2(1+\beta(t_k-\tau_{k-1}))+\varphi(t_k-\tau_{k-1})(t_k-\tau_{k-1}).
$$
Applying this inequality sequentially and using the equality $z_0^i=x_0$ we get estimate  (\ref{x_z_estima}) for $i=1,2$. Further, estimate (\ref{sigma_cont_mod}) holds for $i=1,2$, $j=1,2$. The choice of  $z_k^i$ yields that $( Y_{1,k}^i, Y_{2,k}^i)=( Y_{1,k-1}^i, Y_{2,k-1}^i)$, and $(Y_{1,k}^i, Y_{2,k}^i)\in S(t_{k-1},z_{k-1}^1)$ for $k=\overline{1,r+1}$. Also, the construction of the function $\chi_i$ leads to the equality $( Y_{1,0}^i, Y_{2,0}^i)=(\hat{J}_1(t_0,x_0),\hat{J}_2(t_0,x_0))$. Hence, $(\hat{J}_1(t_0,x_0),\hat{J}_2(t_0,x_0))\in S(t_r,z_r^i)=\{(\sigma_1(z_r^i),\sigma_2(z_r^i))\}$.
By (\ref{sigma_cont_mod}) we conclude that equality (\ref{J_hat}) holds.

Now let us prove that for any position $(t_0,x_0)\in G$ the following inequality is fulfilled:
\begin{equation}\label{J_2_hat_exceed}
    \hat{J}_2(t_0,x_0)\\\geq \lim_{\delta\downarrow 0}\sup\{\sigma_2(x^{1}[T,t_0,x_0,U^*,\Delta,v[\cdot]]), d(\Delta)\leq\delta, v[\cdot]\in \mathcal{V}\}.
\end{equation}

As above, let  $\Delta=\{t_k\}_{k=0}^r$ be a partition of the interval $[t_0,T]$, $d(\Delta)\leq\delta$, $x^1[\cdot]=x^1[\cdot,t_0,x_0,U^*,\Delta,v[\cdot]]$. We add the element  $t_{r+1}=t_r=T$ to the partition $\Delta$. Denote $x_k:= x^1[t_k]$. Let us denote the state of the first player's guide at  time   $t_k$  by $w^1_k=(d^1_k,w^{(a),1}_k,w^{(c),1}_k, Y_{1,k}^1, Y_{2,k}^1)$. Further, let  $z^1_k$ be a point chosen by rule (\ref{z_def}) for the first player at time  $t_k$.

The choice of $z_k^1$ (see (\ref{z_def})) and lemma \ref{lm_x_zero_dis} yield the inequality
$$
\|x_k-z_k^1\|^2\leq \|x_{k-1}-z_{k-1}^1\|^2(1+\beta(t_k-t_{k-1}))+\varphi(t_k-t_{k-1})(t_k-t_{k-1}).
$$

Applying this inequality sequentially and using the equality $z_0^1=x_0$  we get estimate (\ref{x_z_estima}) for $i=1$. Therefore, inequality  (\ref{sigma_cont_mod}) is fulfilled for $i=1$, $j=2$. In addition, $ Y_{2,k}^1\geq  Y_{2,k-1}^2$. Indeed, if $z_k^1=w^{(c),1}_k$, then $( Y_{1,k}^1, Y_{2,k}^1)=( Y_{1,k-1}^1, Y_{2,k-1}^1)$. If  $z_k^1=w^{(a),1}_k$, we have that an element $( Y_{1,k}^1, Y_{2,k}^1)$ is chosen so that $ Y_{2,k}^1$ is the minimum of $\{J_2:(J_1,J_2)\in S(t_{k-1},z_{k-1}^1)\}$.

By the construction we have $( Y_{1,k}^1, Y_{1,k}^1)\in S(t_{k-1},z_{k-1}^1)$. Hence, using condition (S\ref{cond_boundary_m}) we obtain that
\begin{equation}\label{sigma_J_ineq}
    \hat{J}_2(t_0,x_0)\geq  Y_{2,r+1}^1=\sigma_2(z_r^1).
\end{equation} Since inequality   (\ref{sigma_cont_mod}) is valid for $i=1$, $j=2$, estimate (\ref{sigma_J_ineq}) yields inequality~(\ref{J_2_hat_exceed}).

Analogously, we get that for any position $(t_0,x_0)\in G$ the  inequality
\begin{equation}\label{J_1_hat_exceed}
    \hat{J}_1(t_0,x_0)\geq \lim_{\delta\downarrow 0}\sup\{\sigma_1(x^{2}[T,t_0,x_0,V^*,\Delta,u[\cdot]]), d(\Delta)\leq\delta, u[\cdot]\in\mathcal{U}\}
\end{equation} is fulfilled.

Equality (\ref{J_hat}) and inequalities  (\ref{J_2_hat_exceed}), (\ref{J_1_hat_exceed}) mean that the pair of strategies $U^*$ and $V^*$ is a Nash equilibrium on $G$. Moreover, the Nash equilibrium payoff at the initial position $(t_0,x_0)\in G$ is equal to $(\hat{J}_1(t_0,x_0),\hat{J}_2(t_0,x_0))$.

\qed

\section{Existence of Multivalued Value Function}\label{sect_existence}
\subsection{Discrete Time Game}\label{sect_discrete}
In order to prove the existence of a multivalued function satisfying conditions (S\ref{cond_boundary_m})--(S\ref{cond_together_stable_m}) we consider the auxiliary discrete time dynamical game. Let $N$ be a natural number, and let $\delta^N:= T/N$ be a time step. We discretize $[0,T]$ by means of the uniform grid $\Delta^N:=\{t_k^N\}_{k=0}^N$ with $t_k^N=k\delta^N$.

Consider the discrete time control system
\begin{equation}\label{sys_discrete}\begin{split}
    \xi^N(t_{k+1}^N)=\xi^N(t_k)+&\delta^N[f(t_k^N,\xi^N(t_k^N),u(t_k^N))+ g(t_k^N,\xi^N(t_k^N),v(t_k^N))],\\  &k=\overline{0,N-1}, \ \ u(t_k^N)\in P,\ \ v(t_k^N)\in Q.
\end{split}\end{equation}

Denote
$$\mathcal{U}^N:= \{u:[0,T]\rightarrow P: u(t)=u_k^N\in P\mbox{ for }t\in [t_k^N,t_{k+1}^N[\hspace{1pt}\}, $$
$$\mathcal{V}^N:= \{v:[0,T]\rightarrow Q: v(t)=v_k^N\in Q\mbox{ for }t\in [t_k^N,t_{k+1}^N[\hspace{1pt}\}. $$

For $t_*\in \Delta^N$, $\xi_*\in\mathbb{R}^n$, $u\in \mathcal{U}^N$, and $v\in \mathcal{V}^N$ let  $\xi^N(\cdot,t_*,\xi_*,u,v):\Delta^N\cap [t_*,T]\rightarrow\mathbb{R}^n$ be a solution of initial value problem (\ref{sys_discrete}), $\xi^N(t_*)=\xi_*$.

First, we shall estimate  $\|\xi^N(t_+,t_*,\xi_*,u,v)-x(t_+,t_*,x_*,u,v)\|$.

Let $G\subset\ir$ be a compact of initial positions. Let $E'\subset \mathbb{R}^n$ be a compact such that $x(t,t_*,x_*,u,v)\in E'$, and $\xi^N(t,t_*,x_*,u,v)\in E'$ for all natural $N$, $(t_*,x_*)\in G$, $t,t_*\in \Delta^N$, $u\in\mathcal{U}^N$, $v\in\mathcal{V}^N$. Set $$K':=\max\{\|f(t,x,u)+g(t,x,v)\|:t\in [0,T],x\in E',\ \ u\in P,\ \ v\in Q\}.$$ Denote by $L'$ the Lipschitz constant of the function $f+g$ on $[0,T]\times E'\times P\times Q$: for all $t\in [0,T]$, $x',x''\in E',$ $u\in P$, $v\in Q$ $$\|f(t,x',u)+g(t,x',v)-f(t,x'',u)-g(t,x'',v)\|\leq L'\|x'-x''\|.$$
Further, set
\begin{equation*}\begin{split}
\varphi'(\delta):= \sup\{\|&f(t',x',u)-f(t'',x'',u)\|+\|g(t',x',v)-g(t'',x'',v)\|:\\
t',t''\in [&0,T],x',x''\in E',\ \ |t'-t'|\leq\delta,
\|x'-x''\|\leq K'\delta,\ \ u\in P,\ \ v\in Q\}.\end{split}
\end{equation*}

\begin{lemma}\label{lm_x_y_distance} If $t_*,t_+\in \Delta^N$, $t_+\geq t_*$, $(t_*,x_*),(t_*,\xi_*)\in G$, $u\in\mathcal{U}^N$, and $v\in\mathcal{V}^N$, then,
\begin{equation}\label{x_y_distance}\begin{split}
\|x(&t_+,t_*,x_*,u,v)-\xi^N(t_+,t_*,\xi_*,u,v)\|\\
&\leq \|x_*-\xi_*\|\exp(2L'(t_+-t_*))+\varphi'(\delta^N)\exp(L'(t_+-t_*)).\end{split} \end{equation}

\end{lemma}
\begin{proof}
Let $m$ and $r$ be natural numbers such that $t_*=t_m^N$, $t_+=t_r^N$.
Denote $x(\cdot):= x(\cdot,t_*,x_*,u,v)$, $x_k:= x(t_k^N,t_*,x_*,u,v)$, $\xi_k:= \xi^N(t_k^N,t_*,\xi_*,u,v)$.
We have that
\begin{equation*}\begin{split}
x_{k+1}=&x_k+\int_{t_{k}^N}^{t_{k+1}^N}[f(t,x(t),u_k)+g(t,x(t),v_k)]d t\\
=&x_k+\delta^N[f(t_k^N,x_k,u_k)+g(t_k^Nx_k,v_k)]\\
&+\int_{t_k^N}^{t_{k+1}^N}[f(t,x(t),u_k)+g(t,x(t),v_k)-f(t_k^N,x_k,u_k)-g(t_k^N,x_k,v_k)]d t.
\end{split}
\end{equation*}
Here $u_k$ and $v_k$ denote the values of $u$ and $v$ on $[t_k^N,t_{k+1}^N[$ respectively.

Further,
$$\|x(t)-x_k\|\leq K'(t-t_k), \ \ t\in [t_{k},t_{k+1}]. $$ Therefore, the following inequality is fulfilled:
\begin{equation*}
    \int_{t_k}^{t_{k+1}}[f(t,x(t),u_k)+g(t,x(t),v_k)-f(t_k^N,x_k,u_k)-g(t_k^N,x_k,v_k)]dt
    \leq
    \delta^N \varphi(\delta^N).
\end{equation*}
Hence,
\begin{equation}\label{x_k_p_via_x_k}
\|x_{k+1}-x_k-\delta^N[f(t_k^N,x_k,u_k)+g(t_k^N,x_k,v_k)]\|\leq \delta^N \varphi(\delta^N).
\end{equation}

Further, we have
\begin{equation*}\begin{split}
&x_k+\delta^N[f(t_k^N,x_k,u_k)+g(t_k^N,x_k,v_k)]-\xi_{k+1}\\
&=x_k-\xi_k+\delta^N[f(t_k^N,x_k,u_k)+g(t_k^N,x_k,v_k)- f(t_k^N,\xi_k,u_k)-g(t_k^N,\xi_k,v_k)].\end{split}  \end{equation*}
Consequently,
$$\|x_k+\delta^N[f(t_k^N,x_k,u_k)+g(t_k^N,x_k,v_k)]-\xi_{k+1}\|\leq \|x_k-\xi_k\|+\delta^N 2L'\|x_k-\xi_k\|.  $$
This inequality and estimate  (\ref{x_k_p_via_x_k}) yield that
$$\|x_{k+1}-\xi_{k+1}\|\leq \|x_k-\xi_k\|+\delta^N 2L\|x_k-\xi_k\|+\delta^N\varphi(\delta^N). $$
Applying the last inequality sequentially we get inequality (\ref{x_y_distance}).

\end{proof}

Now let us proof the existence of a function satisfying discrete time analogs of conditions  (S1)--(S4).
\begin{theorem}\label{th_disrcrete}
For any natural $N$ there exists an upper semicontinuous multivalued function  $Z^N:\dnr\rightrightarrows\mathbb{R}^2$ satisfying the following properties
\begin{enumerate}
  \item $Z^N(T,\xi)=\{(\sigma_1(\xi),\sigma_2(\xi))\}$;
  \item for all  $(t_*,\xi_*)\in \dnr$, $u\in  P$, $( Y_1, Y_2)\in Z^N(t_*,\xi_*)$ and $t_+\in\Delta^N$, $t_+>t_*$ there exist a control $v\in \mathcal{V}^N$ and a pair $( Y'_1, Y'_2)\in Z^N(t_+,\xi^N(t_+,t_*,\xi_*,u,v))$ such that $ Y_1\geq  Y_1'$;
  \item for all $(t_*,\xi_*)\in \dnr$, $v\in  Q$, $( Y_1, Y_2)\in Z^N(t_*,\xi_*)$ and $t_+\in\Delta^N$, $t_+>t_*$ there exist a control $u\in \mathcal{V}^N$ and a pair $( Y''_1, Y''_2)\in Z^N(t_+,\xi^N(t_+,t_*,\xi_*,u,v))$ such that $ Y_2\geq  Y_2''$;
  \item for all  $(t_*,\xi_*)\in \dnr$, $( Y_1, Y_2)\in Z^N(t_*,\xi_*)$ and $t_+\in\Delta^N$, $t_+>t_*$ there exist controls $u\in \mathcal{U}^N$ and $v\in  \mathcal{V}^N$ such that  $( Y_1, Y_2)\in Z^N(t_+,\xi^N(t_+,t_*,\xi_*,u,v))$.
\end{enumerate}
\end{theorem}
\begin{proof}
In the proof we fix the  number $N$ and omit the superindex $N$. Denote $$ f_k(z,u):=  \delta f(t_k,z,u), \ \ g_k(z,v):= \delta g(t_k,z,v). $$

The proof is by inverse induction on $k$. For  $k=N$ put $Z(t_N,z):= \{\sigma_1(z),\sigma_2(z)\}$.

Now let  $k\in \overline{0,N-1}$. Assume that the values $Z(t_{k+1},z),\ldots, Z(t_N,z)$ are constructed for all $z\in\mathbb{R}^n$. In addition, suppose that the functions $Z(t_{k+1},\cdot),\ldots, Z(t_N,\cdot)$ are upper semicontinuous. Define $$\varsigma^i_{k+1}(z):=\min\{ Y_i:( Y_1, Y_2)\in Z(t_{k+1},z)\}, \ \ i=1,2.$$
It follows from the upper semicontinuity of the multivalued function $Z(t_{k+1},\cdot)$ that the functions $\varsigma^1_{k+1}$ and $\varsigma^2_{k+1}$ are lower semicontinuity.

Set
$$W_k(z):= \bigcup_{u\in  P,v\in Q}Z(t_{k+1},\xi(t_{k+1},t_k,z,u,v)), $$
\begin{equation}\label{varrho_1_def}
\varrho_k^1(z):=\max_{u\in P}\min_{v\in Q}\varsigma_{k+1}^1(\xi(t_{k+1},t_k,z,u,v)),
\end{equation}
\begin{equation}\label{varrho_2_def}
\varrho_k^2(z):=\max_{v\in Q}\min_{u\in P}\varsigma_{k+1}^2(\xi(t_{k+1},t_k,z,u,v)).
\end{equation}
We claim that the multivalued function  $W_k$ is upper semicontinuous. Indeed, let $z^l\rightarrow z^*$, and let $( Y_1^l, Y_2^l)\in W_k(z^l)$ be such that $( Y^l_1, Y^l_2)\rightarrow ( Y_1^*, Y^*_2)$. We have that $( Y^l_1, Y^l_2)\in Z(t_{k+1},\xi(t_{k+1},t_k,z^l,u^l,v^l))$ for some $u^l\in P$, $v^l\in Q$. We can assume without loss of generality that $(u^l,v^l)\rightarrow (u^*,v^*)$. By the continuity of the functions $f_k$ and $g_k$  we get that $\xi(t_{k+1},t_k,z^l,u^l,v^l)=z^l+f_k(z^l,u^l)+g_k(z^l,v^l)\rightarrow \xi(t_{k+1},t_k,z^*,u^*,v^*)$, as $l\rightarrow\infty$. The upper semicontinuity of the multivalued function $Z(t_{k+1},\cdot)$ yields that $( Y^*_1, Y^*_2)\in Z(t_{k+1},\xi(t_{k+1},t_k,z^*,u^*,v^*))\subset W_k(z^*)$.

Now let us show that the functions $\varrho^i_k$ are lower semicontinuous. We give the proof  only for the case $i=1$. For a fixed  $u\in P$ consider the function $z\mapsto \min_{v\in Q}\varsigma_{k+1}^1(\xi(t_{k+1},t_k,z,u,v))$.  We shall prove that this function is lower semicontinuous, i.e. for any $z^*$ the following inequality holds:
\begin{equation}\label{semi_sigma_u}
    \liminf_{z\rightarrow z^*}\min_{v\in Q}\varsigma_{k+1}^1(\xi(t_{k+1},t_k,z,u,v))\geq \min_{v\in Q}\varsigma_{k+1}^1(\xi(t_{k+1},t_k,z^*,u,v)).
\end{equation}
Let $\{z^l\}_{l=1}^\infty$ be a minimizing sequence:
$$\liminf_{z\rightarrow z^*}\min_{v\in Q}\varsigma_{k+1}^1(\xi(t_{k+1},t_k,z,u,v))=\lim_{l\rightarrow \infty}\min_{v\in Q}\varsigma_{k+1}^1(\xi(t_{k+1},t_k,z^l,u,v)). $$ Let $v^l\in Q$ satisfy the condition $$\varsigma^{k+1}_1(\xi(t_{k+1},t_k,z^l,u,v^l))=\min_{v\in Q} \varsigma_{k+1}^1(\xi(t_{k+1},t_k,z^l,u,v)). $$ Hence we have
 \begin{equation}\label{liminf_repres}
 \liminf_{z\rightarrow z^*}\min_{v\in Q}\varsigma_{k+1}^1(\xi(t_{k+1},t_k,z,u,v))=\lim_{l\rightarrow \infty}\varsigma_{k+1}^1(\xi(t_{k+1},t_k,z^l,u,v^l)).
\end{equation} We can assume without loss of generality that the sequence $\{v^l\}$ converges to a control $v^*\in Q$.
From continuity of the function  $\xi(t_{k+1},t_k,\cdot,u,\cdot)$ and lower semicontinuity of the function $\varsigma_{k+1}^1$ we obtain that
\begin{equation*}
\lim_{l\rightarrow \infty}\varsigma_{k+1}^1(\xi(t_{k+1},t_k,z^l,u,v^l))\geq \varsigma_{k+1}^1(\xi(t_{k+1},t_k,z^*,u,v^*))
\geq \min_{v\in Q}\varsigma_{k+1}^1(\xi(t_{k+1},t_k,z^*,u,v)).
\end{equation*}
This inequality and equality  (\ref{liminf_repres}) lead inequality (\ref{semi_sigma_u}).

Since the functions  $z\mapsto \min_{v\in Q}\varsigma^{k+1}_1(\xi(t_{k+1},t_k,z,u,v))$ are lower semicontinuous for each $u\in P$,  the function
$$\varrho_1^k(z)=\max_{u\in P}\min_{v\in Q}\varsigma^{k+1}_1(\xi(t_{k+1},t_k,z,u,v)) $$ is lower semicontinuous.

Put
\begin{equation}\label{Z_t_k_def}
Z(t_k,z):= \{( Y_1, Y_2)\in W^k(z): Y_i\geq \varrho_k^i(z), \ \ i=1,2\}.
\end{equation}

First, we shall prove that it is nonempty. Let $z\in\mathbb{R}^n$. Let  $u_*$ maximize the right-hand side of  (\ref{varrho_1_def}), and let $v_*$ maximize the right-hand side of (\ref{varrho_2_def}).  Choose $( Y_1, Y_2)\in Z(t_{k+1},\xi(t_{k+1},t_k,z,u_*,v_*))$. We have that $( Y_1, Y_2)\in W_k(z)$. Further,
$$\varrho^i_k(z)\leq \varsigma_{k+1}^i(\xi(t_{k+1},t_k,z,u_*,v_*))\leq  Y_i. $$ Therefore, $( Y_1, Y_2)\in Z(t_k,z)$.

The upper semicontinuity of the function $Z(t_k,\cdot)$ follows from (\ref{Z_t_k_def}), the upper semicontinuity of the  multivalued function $W^k$, and the lower semicontinuity of the function $\varrho_k^i(z)$.

Now let us show that the function $Z$ satisfies condition 1--4 of the theorem.

Note that conditions 1 and 4 are fulfilled by the construction. Prove  conditions  2 and 3. Let $(t_*,\xi_*)\in\dnr$, $t_+\in\Delta^N$, $t_+>t$, $u_*\in P$, $( Y_1, Y_2)\in Z(t_*,\xi_*)$. It suffices to consider the case  $t=t_k$, $t_+=t_{k+1}$. By construction of the function $Z$ we have that $ Y_1\geq \varrho_k^1(\xi_*)$. From the definition of the function $\varrho_k^1$ (see (\ref{varrho_1_def})) it follows that
$$ Y_1\geq\max_{u\in P}\min_{v\in Q}\varsigma_{k+1}^1(\xi(t_{k+1},t_k,\xi_*,u,v))\geq \min_{v\in Q}\varsigma_{k+1}^1(\xi(t_{k+1},t_k,\xi_*,u_*,v)). $$ Let $v_*\in Q$ be a control of player II such that
$$\min_{v\in Q}\varsigma_{k+1}^1(\xi(t_{k+1},t_k,\xi_*,u_*,v))=\varsigma_{k+1}^1(\xi(t_{k+1},t_k,\xi_*,u_*,v_*)). $$ From the definition of the function $\varsigma_{k+1}^1$  we get that there exists a pair $( Y_1', Y_2')\in Z(t_{k+1},\xi(t_{k+1},t_k,\xi_*,u_*,v_*))$ such that $ Y_1'=\varsigma_{k+1}^1(\xi(t_{k+1},t_k,\xi_*,u_*,v_*))$. Consequently, $ Y_1\geq  Y_1'$. Hence, condition 2 holds. Condition 3 is proved analogously.

\end{proof}

\subsection{Continuous Time Dynamics}

\begin{theorem}\label{th_existence}
There exists an upper semicontinuous multivalued function   $S:\ir\rightrightarrows\mathbb{R}^2$ with nonempty images satisfying conditions (S1)--(S4).
\end{theorem}

The proof of Theorem \ref{th_existence} is given in the end of the section.

First, for each  $N$ define the multivalued function $S^N:\ir\rightrightarrows\mathbb{R}^2$ by the following rule:
\begin{equation}\label{S_N_def}
    S^N(t,x):=\left\{\begin{array}{lr}
                       Z^N(t_k^N,x), & t\in (t_{k-1},t_{k}),\ \ k=\overline{1,N-1}  \\
                       Z^N(t_k,x)\cup Z^N(t_{k+1},x), & t=t_k,\ \ k=\overline{0,N-1} \\
                       Z^N(t_N^N,x), & t=T
                     \end{array}\right.
\end{equation}
The functions $S^N$ have the closed graph.

Denote $$B(\nu):=\{x:\|x\|\leq \nu\}.$$ For $\Sigma:\ir\rightrightarrows \mathbb{R}^2$ set
$$\gr{\nu} \Sigma:=\{(t,x, Y_1, Y_2):\|x\|\leq \nu, ( Y_1, Y_2)\in \Sigma(t,x)\}. $$

The sets $\gr{\nu}S^N$ are compacts. Indeed,  $$M_{i,\nu}:=\max\{|\sigma_i(x(T,t_*,x_*,u,v))|:t_*\in [0,T],\|x_*\|\leq \nu, u\in\mathcal{U},v\in\mathcal{V}\}<\infty.$$ We have that $\gr{\nu} S_N\subset [0,T]\times B(\nu)\times [-M_{1,\nu},M_{1,\nu}]\times [-M_{2,\nu},M_{2,\nu}].$

Consider the Hausdorff distance between compact sets ${A,B\subset\ir\times\mathbb{R}^2}$
\begin{equation*}\begin{split}
&h(A,B)\\
&:= \max\left\{\max_{(t,x, Y_1, Y_2)\in A}\mathrm{d}((t,x, Y_1, Y_2),B),\max_{(t,x, Y_1, Y_2)\in B}\mathrm{d}((t,x, Y_1, Y_2),A)\right\}.\end{split}
\end{equation*}

Here $\mathrm{d}((t,x, Y_1, Y_2),A)$ is the distance from the point $(t,x, Y_1, Y_2)$ to the set $A$ generated by the norm
$$\|(t,x, Y_1, Y_2)\|=|t|+\|x\|+| Y_1|+| Y_2|. $$

Since for any  $\nu$ the set $[0,T]\times B(\nu+1)\times [-M_{1,\nu},M_{1,\nu}]\times [-M_{2,\nu},M_{2,\nu}]$ is  compact, using  \cite[Theorem 4.18]{RockWets} we get that one can extract a convergent subsequence from the sequence $\{\grp{\nu} S^N\}_{N=1}^\infty$.

Using the diagonal process we construct the subsequence $\{N_j\}$ such that for any  $\nu$ there exists the limit
$$\lim_{j\rightarrow\infty}\grp{\nu}S^{N_j}=R_\nu.$$ One can choose the subsequence  $\{N_j\}$ satisfying the property:
$$h(\grp{\nu}S^{N_j},R_\nu)\leq 2^{-j}\mbox{ for }j\geq \nu.$$ Denote $\widetilde{S}_j:= S^{N_j}$.

\begin{lemma}\label{lm_plus}
Let $( Y_{1,l}, Y_{2,l})\in \widetilde{S}_{j_l}(t_l,x_l)$, $\|x_l\|\leq \nu+1$,  $(t_l,x_l)\rightarrow (t^*,x^*)$, $( Y_{1,l}, Y_{1,l})\rightarrow ( Y_1^*, Y_2^*)$, as $l\rightarrow\infty$. Then $(t^*,x^*, Y_1^*, Y_2^*)\in R_\nu$.
\end{lemma}
\begin{proof}
Consider the set $R_\nu\cup \{(t^*,x^*, Y_1^*, Y_2^*)\}$. This set is closed. We claim that
\begin{equation}\label{R_nu_limits}
h(\grp{\nu}\widetilde{S}_{j_l},R_\nu\cup\{(t^*,x^*, Y_1^*, Y_2^*)\})\rightarrow 0,\ \  l\rightarrow\infty.
\end{equation}
Indeed, $\mathrm{d}((t,x,Y_1,Y_2),R_\nu\cup\{(t^*,x^*, Y_1^*, Y_2^*)\})\leq \mathrm{d}((t,x,Y_1,Y_2),R_\nu)$ for all $(t,x,Y_1,Y_2)\leq \grp{\nu}\widetilde{S}_{j_l}$.
Hence
\begin{equation}\label{converge_haus_1}
\max_{(t,x,Y_1,Y_2)\in \grp{\nu}\widetilde{S}_{j_l}}\mathrm{d}((t,x,Y_1,Y_2),R_\nu\cup\{(t^*,x^*, Y_1^*, Y_2^*)\})\rightarrow 0, \mbox{ as }l\rightarrow \infty.
\end{equation}

Further,  the following convergence is valid: $$\max_{(t,x, Y_1, Y_2)\in R_\nu\cup \{(t^*,x^*, Y_1^*, Y_2^*)\}}\{\mathrm{d}((t,x, Y_1, Y_2),\grp{\nu}\widetilde{S}_{j_l})\}\rightarrow 0, \mbox{ as } l\rightarrow\infty. $$
This and (\ref{converge_haus_1}) yield (\ref{R_nu_limits}).

Formula (\ref{R_nu_limits}) means that
$$R_\nu\cup\{(t^*,x^*, Y_1^*, Y_2^*)\}=\lim_{l\rightarrow\infty}\grp{\nu}\widetilde{S}_{j_l}=R_\nu. $$ This completes the proof. 
\end{proof}

\begin{lemma}\label{lm_R_repres} For $r>\nu$ the following equality holds: $$R_{r}\cap ([0,T]\times B(\nu)\times\mathbb{R}^2)=R_\nu\cap ([0,T]\times B(\nu)\times\mathbb{R}^2). $$
\end{lemma}
\begin{proof}
Let $(t,x, Y_1, Y_2)\in R_r$, $\|x\|\leq \nu$, and $j\geq r$. There exists a quadruple $(\theta_j,y_j,\zeta_{1,j},\zeta_{2,j})\in \grp{r}\widetilde{S}_j$ such that
\begin{equation}\label{point_s_rep}
    |t-\theta_j|+\|x-y_j\|+| Y_1-\zeta_{1,j}|+| Y_2-\zeta_{2,j}|=\mathrm{d}((t,x, Y_1, Y_2),\grp{r}\widetilde{S}_j)\leq 2^{-j}.
\end{equation}
Therefore,  $\|x-y_j\|\leq\mathrm{d}((t,x, Y_1, Y_2),\grp{r}\widetilde{S}_j)\leq 2^{-j}$. We have that $\|y_j\|\leq \|x\|+2^{-j}\leq \nu+1$. Therefore,  $(\theta_j,y_j,\zeta_{1,j},\zeta_{2,j})\in \grp{\nu}\widetilde{S}_j$. It follows from formula (\ref{point_s_rep}) and lemma  \ref{lm_plus}  that $(t,x, Y_1, Y_2)\in R_\nu$. Since the quadrable $(t,x, Y_1, Y_2)$ satisfies the condition  $\|x\|\leq \nu$ we conclude that
$$R_{r}\cap ([t_0,T]\times B(\nu)\times\mathbb{R}^2)\subset R_\nu\cap ([t_0,T]\times B(\nu)\times\mathbb{R}^2). $$ The opposite inclusion is proved in the same way.

\end{proof}

Define the multivalued function $\bar{S}:\ir\rightrightarrows\mathbb{R}^2$ by the following rule: for $\|x\|\leq\nu$
$$\bar{S}(t,x):=\{( Y_1, Y_2):(t,x, Y_1, Y_2)\in R_\nu\}.$$ Note that this definition is correct by virtue of lemma \ref{lm_R_repres}.
We have that $\gr{\nu}\bar{S}=R_\nu\cap ([t_0,T]\times B(\nu)\times\mathbb{R}^2)$.  

\vspace{5pt}
\noindent\textit{Proof of theorem \ref{th_existence}}
We shall show that the function  $\bar{S}$ has nonempty images, and satisfies conditions (S\ref{cond_boundary_m})--(S\ref{cond_together_stable_m}).

First we shall prove that the sets $\bar{S}(t,x)$ are nonempty. Let $\nu$ satisfy the condition $\|x\|< \nu$, and let $( Y_{1,j}, Y_{2,j})\in \widetilde{S}_{j}(t,x)$. Since $\widetilde{S}_{j}(t,x)\subset [-M_{1,\nu},M_{1,\nu}]\times [-M_{2,\nu},M_{2,\nu}]$, there exists a subsequence $\{( Y_{1,j_l}, Y_{2,j_l})\}_{l=1}^\infty$ converging to  a pair $( Y_1^*, Y_2^*)$. By lemma \ref{lm_plus} we obtain that $( Y_1^*, Y_2^*)\in\bar{S}(t,x)$.

Now let us prove that the multivalued function $\bar{S}$ satisfies conditions (S1)--(S4).

We begin with condition (S1). Let $x_*\in\mathbb{R}^n$. Choose $\nu$  such that the following conditions hold
\begin{enumerate}
\item $x(t,T,x_*,u,v)\in B(\nu)$ for all $t\in [0,T]$, $u\in \mathcal{U}$, $v\in\mathcal{V}$;
\item  all  $z$ such that $x_*=\xi^N(T,t,z,u,v)$ for some natural $N$, $t\in \Delta^N$, $u\in\mathcal{U}^N$, $v\in\mathcal{V}^N$ belong to $B(\nu)$.
\end{enumerate}
Let $K_\nu$ be defined by (\ref{K_def}) for $E=B(\nu+1)$.

Let $N$  be a natural number,  $t_*\in\Delta^N$, and  $\xi_*\in B(\nu)$. By conditions~1 and~4 of  Theorem \ref{th_disrcrete} we have that if $( Y_1, Y_2)\in Z^N(t_*,\xi_*)$, then there exist  $u\in\mathcal{U}^N$, and $v\in\mathcal{V}^N$ such that
\begin{equation}\label{J_i_sigma}
 Y_i=\sigma_i(\xi^N(T,t_*,\xi_*,u,v)),  \ \ i=1,2.
\end{equation}
We have  the estimate
\begin{equation}\label{y_N_y_estima}
    \|\xi_*-\xi^N(T,t_*,\xi_*,u,v)\|\leq K_\nu(T-t_*).
\end{equation}

Let $(J_1,J_2)\in \bar{S}(T,x)$. This means that there exists a sequence  $\{(t_j,x_j, Y_{1,j}, Y_{2,j})\}_{j=1}^\infty$ such that $( Y_{1,j}, Y_{2,j})\in \widetilde{S}_{j}(t_j,x_j)=S^{N_{j}}(t_j,x_j)$, and $t_j\rightarrow T$, $x_j\rightarrow x$, $ Y_{i,j}\rightarrow J_i$ as $j\rightarrow\infty$. Let $\theta_j\in\Delta^{N_j}$ be such that $( Y_{1,j}, Y_{2,j})\in Z^{N_j}(\theta_j,x_j)$ and $t_j\in (\theta_j-\delta^N,\theta_j]$.   Combining this, (\ref{J_i_sigma}), and (\ref{y_N_y_estima}) we conclude that for any $j$ there exists $x_j'\in B(\nu)$ such that $\|x_j-x_j'\|\leq K_\nu(T-t_j)$ and $ Y_{i,j}=\sigma_i(x_j')$, $i=1,2$. We have that $x_j'\rightarrow x$, as $j\rightarrow\infty$. By the continuity of the functions $\sigma_i$ we obtain that
$$J_i=\lim_{l\rightarrow\infty} Y_{i,j}=\lim_{j\rightarrow\infty} \sigma_i(x_j')=\sigma_i(x).$$

Now we shall prove the fulfillment of condition (S2). Let $(t_*,x_*)\in \ir$, $(J_1,J_2)\in \bar{S}(t_*,x_*)$, $u\in P$, $t_+\in [t_*,T]$. We shall show that there exists $y^2(\cdot)\in {\rm Sol}^2(t_*,x_*,u)$ such that
$J_1'\leq J_1$ for some  $(J_1',J_2')\in \bar{S}(t_+,y^2(t_+))$.

There exists a sequence $\{(t_j,x_j, Y_{1,j}, Y_{2,j})\}_{j=1}^\infty$ such that $( Y_{1,j}, Y_{2,j})\in \widetilde{S}_{j}(t_j,x_j)=S^{N_{j}}(t_j,x_j)$, and $t_j\rightarrow t_*$, $x_j\rightarrow x_*$, $ Y_{i,j}\rightarrow J_i$, as $j\rightarrow\infty$. Let $\theta_j$ be an element of $\Delta^{N_j}$ such that $( Y_{1,j}, Y_{2,j})\in Z^{N_j}(\theta_j,x_j)$ and $t_j\in (\theta_j-\delta^N,\theta_j]$. Further, let $\tau_j$ be the least element of $\Delta^{N_j}$ such that $t_+\leq \tau_j$.

By condition 2 of Theorem \ref{th_disrcrete} for each $j$ there exist a control $v_j\in \mathcal{V}^{N_j}$, and a pair $( Y_{1,j}', Y_{2,j}')$ such that $( Y_{1,j}', Y_{2,j}')\in Z^{N_j}(\tau_j,\xi^{N_j}(\tau_j,\theta_j,x_j,u,v_j))\subset \widetilde{S}^j(\tau_j,\xi^{N_j}(\tau_j,\theta_j,x_j,u,v_j))$ and $ Y_{1,j}'\leq Y_{1,j}$. 
By lemma \ref{lm_x_y_distance} we have that $$\|x(\tau_j,\theta_j,x_j,u,v_j)-\xi^{N_j}(\tau_j,\theta_j,x_j,u,v_j)\|\leq \varphi'(\delta^{N_j})\exp(LT).$$
We may extract a subsequence $\{j_l\}_{l=1}^\infty$ such that $\{x(\cdot,\theta_{j_l},x_{j_l},u,v_{j_l})\}_{l=1}^\infty$ converges to some motion $y^2(\cdot)$, and $\{( Y_{1,j_l}', Y_{2,j_l}')\}$ converges to some pair $(J_1',J_2')$. We have that $y^2(\cdot)\in {\rm Sol}^2(t_*,x_*,u)$. Lemma \ref{lm_plus} 
gives the inclusion $(J_1',J_2')\in \overline{S}(t_+,y^2(t_+))$. We also have $$J_1'\leq J_1. $$ This completes the proof of condition (S2).

Conditions  (S3) and (S4) are proved analogously.

\qed

\section{Conclusion}
In this paper the Nash equilibria for differential games in the class of control with guide strategies are constructed on the basis of  an upper semicontinuous  multivalued function satisfying boundary condition  and some viability conditions. 
The main result is that for any compact of initial positions and any selector of the multivalued map it is possible to construct a Nash equilibrium such that the corresponding  players' payoff  is equal to the value of the given selector. The existence of the multivalued function satisfying proposed conditions is also proved.
If the upper semicontinuous multivalued function is replaced with  a continuous function, then the construction of the strategies is simplified. However, in the general case the desired continuous function doesn't exist.

Only  two players nonzero-sum differential games with terminal payoffs and compact control spaces were considered. The results can be extended to the games with payoffs equal to the sum of terminal and running parts by introducing new variables describing running payoffs. Note that if the running payoff of each player  doesn't depend on the control of the another one, then the players need only the information about the state variable to construct the Nash equilibrium control with guide strategies. The condition of compactness of control spaces is essential, and the methods developed in the paper can't be used for the games with unbounded control spaces. (Such games were studied by Bressan and Shen in \cite{Bres2}, \cite{Bres3} on the basis of BV solutions of PDEs.)

Future work includes the extension of the obtained results to the game with many players and the stability analysis of proposed conditions.



\end{document}